\documentclass[amsthm,leqno]{monsky2009} 

\usepackage{graphicx}

\usepackage{amssymb}
\usepackage{amsmath}

\usepackage{bm}

\newtheorem*{ctheorem}{Classification Theorem}
\newtheorem*{corollary*}{Corollary}
\newtheorem*{remark*}{Remark}
\newtheorem*{remarks*}{Remarks}
\newtheorem*{proofidea*}{Idea of proof}
\newtheorem*{speculation*}{Speculation}

\newtheorem{definition}{Definition}[section] 

\newtheorem{lemma}[definition]{Lemma}
\newtheorem{proposition}[definition]{Proposition}
\newtheorem{corollary}[definition]{Corollary}

\newcommand{\newz}{\ensuremath{\mathbb{Z}}}
\newcommand{\newi}{\ensuremath{\bm{\mathrm{I}}}}

\begin{document}

\begin{frontmatter}






\title{Frobenius' result on simple groups of order $\bm{\frac{p^{3}-p}{2}}$}
\author{Paul Monsky}

\address{Brandeis University, Waltham MA  02454-9110, USA\\  monsky@brandeis.edu }

\begin{abstract}
The complete list of pairs of non-isomorphic finite simple groups having the same order is well-known. In particular for $p>3$, $PSL_{2}(\newz/p)$ is the ``only'' simple group of order $\frac{p^{3}-p}{2}$. It's less well-known that Frobenius proved this uniqueness result in 1902. This note presents a version of Frobenius' argument that might be used in an undergraduate honors algebra course. It also includes a short modern proof, aimed at the same audience, of the much earlier result that $PSL_{2}(\newz/p)$ is simple for $p>3$; a result stated by Galois in 1832.
\end{abstract}
\maketitle

\end{frontmatter}


\section{Background}
\label{section1}

Let $p$ be a prime and $SL_{2}(\newz/p)$ be the group of 2 by 2 determinant 1 matrices with entries in $\newz/p$. The quotient, $PSL_{2}(\newz/p)$, of $SL_{2}(\newz/p)$ by $\{\pm \newi\}$ is for $p>2$ a group of order $\frac{p^{3}-p}{2}$. Galois \cite{2} introduced and studied this group; in his 1832 letter to Auguste Chevalier he says that it is easily shown to be simple for $p>3$. (There are many proofs of simplicity. I'll give a short one in section \ref{section6}.) In 1902 Frobenius \cite{1} classified certain transitive permutation groups on $p+1$ letters up to permutation isomorphism, and deduced as a corollary that $PSL_{2}(\newz/p)$ is the ``only'' simple group of order $\frac{p^{3}-p}{2}$.

Frobenius' proof of this very early result in the classification of the finite simple groups, though elementary, isn't well-known and hasn't found its way into textbooks. In this note I give a version of it, based on Sylow theory and the cyclicity of $(\newz/p)^{*}$. This version could perhaps be presented in an undergraduate honors algebra course. I thank Jim Humphreys for his close reading of this note, his encouragement, and his expository suggestions.

Another description of $PSL_{2}(\newz/p)$ will be useful. Let $V$ be the space of column vectors, $\left(\begin{smallmatrix}x\\y\end{smallmatrix}\right)$, with entries in $\newz/p$. $SL_{2}(\newz/p)$ acts on the set consisting of the $p+1$ one-dimensional subspaces of $V$. We identify this space with $\newz/p\cup\{\infty\}$ as follows. Given a subspace with generator $\left(\begin{smallmatrix}x\\y\end{smallmatrix}\right)$, map it to the element $z=\frac{x}{y}$ of $\newz/p\cup\{\infty\}$. Since $\left(\begin{smallmatrix}x\\y\end{smallmatrix}\right)$ is mapped to $\left(\begin{smallmatrix}x+y\\y\end{smallmatrix}\right)$ by $\left(\begin{smallmatrix}1 & 1\\0 & 1\end{smallmatrix}\right)$ and to $\left(\begin{smallmatrix}-y\\x\end{smallmatrix}\right)$ by $\left(\begin{smallmatrix}0 & -1\\1 & 0\end{smallmatrix}\right)$, the images of $\left(\begin{smallmatrix}1 & 1\\0 & 1\end{smallmatrix}\right)$ and $\left(\begin{smallmatrix}0 & -1\\1 & 0\end{smallmatrix}\right)$ are the translation $z\rightarrow z+1$, and the involution $z\rightarrow -\frac{1}{z}$. Easy arguments with row and column operations show that $\left(\begin{smallmatrix}1 & 1\\0 & 1\end{smallmatrix}\right)$ and $\left(\begin{smallmatrix}0 & -1\\1 & 0\end{smallmatrix}\right)$ generate $SL_{2}(\newz/p)$. Since the kernel of the action of $SL_{2}(\newz/p)$ consists of $\newi$ and $-\newi$, $PSL_{2}(\newz/p)$ identifies with the transitive group of permutations of $\newz/p\cup\{\infty\}$ generated by $z\rightarrow z+1$ and $z\rightarrow -\frac{1}{z}$.

I'll prove the following (version of a) result of Frobenius, and its easy corollary:

\begin{ctheorem}
Let $p\ne 2$ be prime and $G$ be a transitive group of permutations of $\newz/p\cup\{\infty\}$. Suppose $|G|=\frac{p^{3}-p}{2}$, and that $G$ contains the translations. Then one of the following holds:
\begin{enumerate}
\item[(a)] $z\rightarrow -\frac{1}{z}$ is in $G$. (In this case the description of $PSL_{2}(\newz/p)$ given above and the fact that $|G|=|PSL_{2}(\newz/p)|$ tell us that $G$ is generated by $z\rightarrow z+1$ and $z\rightarrow -\frac{1}{z}$, and is permutation-isomorphic to $PSL_{2}(\newz/p)$ in its action on the 1-dimensional subspaces of $V$.)
\item[(b)] $p=7$ and $G$ contains the involution $(0 \infty)(1 3)(2 6)(4 5)$ or the involution $(0 \infty)(1 5)(2 3)(4 6)$.  In these cases $G$ is generated by $z\rightarrow z+1$, $z\rightarrow 2z$, and the involution, and has a normal subgroup of order 8.
\end{enumerate}
\end{ctheorem}

\begin{corollary*}
When $p>3$, $PSL_{2}(\newz/p)$ is, up to isomorphism, the only simple group of order $\frac{p^{3}-p}{2}$.
\end{corollary*}

The theorem is trivial when $p=3$. For now $|G|=12$ and $G$ is a permutation group on 4 elements. So $G$ consists of the even permutations, and thus contains $(0 \infty)(1 2)$. In the following sections we prove the theorem for $p>3$, but now we show how the corollary follows. Suppose $G$ is a simple group of order $\frac{p^{3}-p}{2}$ with $p>3$. Then $p$ divides $|G|$ and $G$ has $mp+1$ $p$-Sylow subgroups; since $G$ is simple $m>0$. Furthermore $\frac{p^{2}-1}{mp+1}$ is an integer $\equiv -1\pod{p}$ and so is $\ge p-1$; it follows that $m=1$. $G$ acts on the set $S$ consisting of the $p+1$ $p$-Sylows, and by Sylow theory the action is transitive. Since $G$ is simple, the action is faithful. Select an element $\sigma$ of $G$ of order $p$. This element acts by a $p$-cycle on $S$; denote the element it fixes by $\infty$. We may label the remaining elements of $S$ with tags in $\newz/p$ so that $\sigma$ is the translation $(0 1 \cdots p-1)$, $z\rightarrow z+1$, of $\newz/p\cup\{\infty\}$. If we view $G$ as a group of permutations of $\newz/p\cup\{\infty\}$, the hypotheses of the classification theorem are satisfied. Since $G$ has no normal subgroup of order 8, we're in the situation of (a), and we conclude that $G$ is isomorphic to $PSL_{2}(\newz/p)$.

\section{Easy facts about $\bm{G}$}

For the rest of this note $G$ is a group satisfying the hypotheses of the classification theorem. Then the stabilizer, $G_{\infty}$, of $\infty$ in $G$ contains the translations and so is transitive on $\newz/p$. Consequently:

\begin{lemma}
\label{lemma2.1}
$G$ is doubly transitive on $\newz/p\cup\{\infty\}$.
\end{lemma}

\begin{definition}
\label{def2.2}
$K$ is the subgroup of $G$ consisting of elements fixing 0 and $\infty$. $\bar{K}$ consists of the elements of $G$ interchanging 0 and $\infty$. $H=K\cup \bar{K}$ is the stabilizer of $\{0,\infty\}$ in $G$.
\end{definition}

Since $G$ is doubly transitive, $|K|=\frac{|G|}{p(p+1)}=\frac{p-1}{2}$. Double transitivity also shows that $\bar{K}$ is non-empty. So $K$ is of index 2 in $H$, and $|\bar{K}|=|K|=\frac{p-1}{2}$.

\begin{definition}
\label{def2.3}
$R$ is the set of squares in $(\newz/p)^{*}$, $N$ the set of non-squares.
\end{definition}

Since $(\newz/p)^{*}$ is cyclic, so is $R$. Furthermore, $|R|=|N|=\frac{p-1}{2}$.

\begin{lemma}
\label{lemma2.4}
\hspace{2em} 
\begin{enumerate}
\item[(1)] $K$ is cyclic and consists of the maps $z\rightarrow az$, $a$ in $R$.
\item[(2)] No element of $G$ fixes more than 2 letters.
\end{enumerate}
\end{lemma}

\begin{proof}
$|G_{\infty}|=\frac{|G|}{p+1}=p\left(\frac{p-1}{2}\right)$. The Sylow theorems then show that the group of translations is the unique $p$-Sylow subgroup of $G_{\infty}$, and so is normal in $G_{\infty}$. So for $\tau$ in $G_{\infty}$, $\tau\circ (z\rightarrow z+1)=(z\rightarrow z+a)\circ\tau$ for some a in $(\newz/p)^{*}$.  If $\tau$ is in $K$, $\tau(0)=0$. Since $\tau(z+1)=\tau{z}+a$, $\tau(z)=az$ for all $z$ in $\newz/p$. Now the maps $z\rightarrow az$, $a$ in $(\newz/p)^{*}$, form a cyclic group of order $p-1$. Since $K$ is a subgroup of that group of order $\frac{p-1}{2}$ we get (1).

Suppose next that $\tau\ne e$ fixes 3 or more letters. By double transitivity we may assume that 2 of these letters are 0 and $\infty$, so that $\tau$ is in $K$. But the only map $z\rightarrow az$ fixing a third letter is $e$.
\end{proof}

\begin{lemma}
\label{lemma2.5}
Suppose $\tau\in\bar{K}$.
\begin{enumerate}
\item[(1)] If $p\equiv 1\pod{4}$, $-1\in R$ and $\tau$ stabilizes $R$ and $N$.
\item[(2)] If $p\equiv 3\pod{4}$, $-1\in N$ and $\tau$ interchanges $R$ and $N$.
\end{enumerate}
\end{lemma}

\begin{proof}
By Lemma \ref{lemma2.4} (1), the orbits of $K$ acting on $(\newz/p)^{*}$ are $R$ and $N$. Since $\tau$ normalizes $K$ it permutes these orbits. Suppose first that $p\equiv 1\pod{4}$. Then $|R|=\frac{p-1}{2}$ is even and $R$ contains an element of $(\newz/p)^{*}$ of order 2, which must be $-1$. Furthermore by Lemma \ref{lemma2.4} (1), $z\rightarrow -z$ is in $G$ and has an orbit $(u,v)$ of size 2. By double transitivity some conjugate, $\lambda$, of this element lies in $\bar{K}$, and, like $z\rightarrow -z$, fixes 2 letters. Such a $\lambda$ cannot possibly interchange $R$ and $N$. So it stabilizes $R$ and $N$, and since $\bar{K}$ is a coset of $K$ in $H$, the same is true of all $\tau$ in $\bar{K}$. Suppose next that $p\equiv 3\pod{4}$. Then $|R|=\frac{p-1}{2}$ is odd, so $-1$ cannot be in $R$. If the lemma fails there is a $\tau$ in $\bar{K}$ with $\tau(1)$ in $R$. Since $\bar{K}$ is a coset of $K$ in $H$ there is a $\lambda$ in $\bar{K}$ with $\lambda(1)=1$. Then $\lambda\circ\lambda$ fixes the letters 0, $\infty$ and 1. By Lemma \ref{lemma2.4} (2), $\lambda$ has order 2 and fixes 1. Since $\lambda$ is a product of disjoint 2-cycles, $\lambda$ must fix a second letter as well. By double transitivity some conjugate of $\lambda$ is an order 2 element of $K$. But $|K|=\frac{p-1}{2}$ is odd.
\end{proof}

\begin{lemma}
\label{lemma2.6}
Suppose $\tau\in\bar{K}$. Then there is an $n$ such that whenever $z$ is in $(\newz/p)^{*}$ and $a$ is in $R$, $\tau(az)=a^{n}\tau(z)$. Furthermore $\frac{p-1}{2}$ divides $n^{2}-1$.
\end{lemma}

\begin{proof}
$\tau$ normalizes $K$. So $\sigma\rightarrow\tau\sigma\tau^{-1}$ is an automorphism of $K$ which is of the form $\sigma\rightarrow\sigma^{n}$ since $K$ is cyclic. Then $\tau\circ(z\rightarrow az)=(z\rightarrow a^{n}z)\circ\tau$, giving the first result. Since the square of the automorphism is the identity, $\frac{p-1}{2}$ divides $n^{2}-1$.
\end{proof}

\begin{remark*}
$n$ is prime to $\frac{p-1}{2}$. So when $p\equiv 1\pod{4}$, $n$ is odd. We're free to modify $n$ by $\frac{p-1}{2}$, and so when $p\equiv 3\pod{4}$ we may (and shall) assume that $n$ is odd as well.
\end{remark*}

\section{The case $\bm{p\equiv 1\pod{4}}$}
\label{section3}

In this section $p\equiv 1\pod{4}$. Our first goal is to show that the $n$ of Lemma \ref{lemma2.6} can be chosen to be $-1$.

\begin{definition}
\label{def3.1}
$X$ is the set of pairs whose first element is a $\tau$ in $G$, and whose second element is a size 2 orbit $\{ u,v\}$ of $\tau$.
\end{definition}

\begin{lemma}
\label{lemma3.2}
$|X|=\left(\frac{p^{2}+p}{2}\right)\left(\frac{p-1}{2}\right)$.
\end{lemma}

\begin{proof}
The number of size 2 subsets of $\newz/p\cup\{\infty\}$ is $\frac{p^{2}+p}{2}$. We prove the lemma by showing that for each such subset $\{ u,v\}$ there are exactly $\frac{p-1}{2}$ elements of $G$ having $\{ u,v\}$ as an orbit. By double transitivity we may assume $\{ u,v\}=\{ 0,\infty\}$. But $\tau$ has $\{ 0,\infty\}$ as an orbit precisely when $\tau$ is in $\bar{K}$.
\end{proof}

\begin{lemma}
\label{lemma3.3}
Every element of $\bar{K}$ has order 2.
\end{lemma}

\begin{proof}
Since $-1$ is in $R$, $\tau : z\rightarrow -z$ is in $K$. The letters fixed by $\tau$ are 0 and $\infty$; it follows that every element of $G$ commuting with $\tau$ stabilizes the set $\{ 0,\infty\}$ and lies in $H$. So the centralizer of $\tau$ in $G$ has order at most $|H|=p-1$, and the number of conjugates of $\tau$  is at least  $\frac{|G|}{p-1}=\frac{p^{2}+p}{2}$. Call these conjugates $\tau_{i}$. Like $\tau$, each $\tau_{i}$ has $\frac{p-1}{2}$ orbits of size 2. So the number of pairs whose first element is some $\tau_{i}$ and whose second is an orbit of that $\tau_{i}$ is at least $\left(\frac{p^{2}+p}{2}\right)\cdot\left(\frac{p-1}{2}\right)$. By Lemma \ref{lemma3.2} these pairs exhaust $X$. Suppose now that $\lambda$ is in $\bar{K}$. Then $\lambda$ has a size 2 orbit, $\{ 0,\infty\}$, and so must be some $\tau_{i}$, proving the lemma.
\end{proof}

\begin{corollary}
\label{corollary3.4}
The $n$ of Lemma \ref{lemma2.6} can be taken to be $-1$.
\end{corollary}

\begin{proof}
Take $\tau$ in $\bar{K}$, $\sigma$ in $K$. By Lemma \ref{lemma3.3}, $(\tau\sigma)(\tau\sigma)=e$ and $\tau=\tau^{-1}$. So $\tau\sigma\tau^{-1}=\sigma^{-1}$. Examining the proof of Lemma \ref{lemma2.6} we get the result.
\end{proof}

\begin{corollary}
\label{corollary3.5}
There is a $\lambda$ in $\bar{K}$ and a $c$ in $(\newz/p)^{*}$ such that:
\begin{enumerate}
\item[(1)] $\lambda(z)=z^{-1}$ for $z$ in $R$.
\item[(2)] $\lambda(z)=cz^{-1}$ for $z$ in $N$.
\end{enumerate}
\end{corollary}

\begin{proof}
By Lemma \ref{lemma2.5} there is a $\bar{K}$ in $R$ with $\lambda(1)=1$. The result now follows from Corollary \ref{corollary3.4} and Lemma \ref{lemma2.6}
\end{proof}

Suppose we can show that the $c$ of Corollary \ref{corollary3.5} is 1. Then, composing $\lambda$ with the element  $z\rightarrow -z$ of $K$ we deduce that $z\rightarrow -\frac{1}{z}$ is in $G$. So to prove the classification theorem for $p\equiv 1\pod{4}$ it's enough to show that $c=1$.

\begin{proposition}
\label{prop3.6}
When $p\equiv 1\pod{4}$, $z\rightarrow -\frac{1}{z}$ is in $G$.
\end{proposition}

\begin{proof}
Let $\alpha(z)=1-\lambda(z)$ with $\lambda$ as in Corollary \ref{corollary3.5}. Since $-1$ is in $R$, $\alpha$ is in $G$. Using the fact that $\lambda\circ\lambda =e$ we see that $\alpha^{-1}(z)=\lambda(1-z)$. Now $\alpha(0)=\infty$, $\alpha(\infty)=1$, $\alpha(1)=1-1=0$. So $\alpha\circ\alpha\circ\alpha$ fixes the letters 0, $\infty$ and 1. By Lemma \ref{lemma2.4} (2), $\alpha$ has order 3.

Since $p\equiv 1\pod{4}$, $p-1$ is in $R$. As not all of $1,2,\ldots ,p-2$ are in $R$ there is an $x$ in $R$ with $x-1$ in $N$. The paragraph above shows that $\alpha(\alpha(x))=\alpha^{-1}(x)=\lambda(1-x)$. We'll use this to show that $c=1$. Since $\lambda(x)=-\frac{1}{x}$, $\alpha(x)=\frac{x-1}{x}$ is in $N$. Consequently, $\lambda(\alpha(x))=\frac{cx}{x-1}$. Then $\alpha(\alpha(x))=\frac{x-1-cx}{x-1}$ while $\lambda(1-x)=-\frac{c}{x-1}$. So $x-1=cx-c$, and $c=1$.
\end{proof}

\section{$\bm{p\equiv 3\pod{4}$}. The main case}
\label{section4}

In this section $p\equiv 3\pod{4}$.

\begin{lemma}
\label{lemma4.1}
There is a unique $\lambda$ in $\bar{K}$ with $-\lambda(1)\lambda(-1)=1$. Furthermore $\lambda$ has order 2.
\end{lemma}

\begin{proof}
Fix $\tau$ in $\bar{K}$. By Lemma \ref{lemma2.5}, $-1$ and $\tau(1)$ are in $N$ while $\tau(-1)$ is in $R$. So $u=-\tau(1)\tau(-1)$ is in $R$. Replacing $\tau$ by $z\rightarrow v\tau(z)$ with $v$ in $R$ multiplies $-\tau(1)\tau(-1)$ by $v^{2}$. Since there is a unique $v$ in $R$ with $v^{2}=u^{-1}$ we get the existence and uniqueness of $\lambda$. By Lemma \ref{lemma2.5} there are $a$ and $b$ in $R$ with $\lambda(a)=-1$, $\lambda(-b)=1$. Taking $n$ as in Lemma \ref{lemma2.6} we find that $a^{n}\lambda(1)=-1$, $b^{n}\lambda(-1)=1$. Multiplying we see that $(ab)^{n}=1$, so $ab=1$. Now $-\lambda^{-1}(-1)\lambda^{-1}(1)=(-a)(-b)=1$, and the uniqueness of $\lambda$ tells us that $\lambda=\lambda^{-1}$.
\end{proof}

\begin{corollary}
\label{corollary4.2}
Choose $n$ odd as in Lemma \ref{lemma2.6}. Then there is a $\lambda$ of order 2 in $\bar{K}$ and a $c$ in $N$ with

\begin{enumerate}
\item[(1)] $\lambda(z)=cz^{n}\qquad z\mbox{ in }R$
\item[(2)] $\lambda(z)=c^{-1}z^{n}\qquad z\mbox{ in }N$
\item[(3)] $c^{n}=c$
\end{enumerate}
\end{corollary}

\begin{proof}
Take $\lambda$ as in Lemma \ref{lemma4.1} and set $c=\lambda(1)$. By Lemma \ref{lemma2.5}, $c$ is in $N$. Since $\lambda(1)=c$, $\lambda(-1)=-\frac{1}{c}=\frac{1}{c}\cdot(-1)^{n}$. Lemma \ref{lemma2.6} then gives (1) and (2). Since $c$ is in $N$, $\lambda(c)=c^{-1}\cdot c^{n}$. But as $\lambda$ has order 2, $\lambda(c)=1$.
\end{proof}

\begin{lemma}
\label{lemma4.3}
Let $\alpha(z)=1-c^{-1}\lambda(z)$ with $c$ and $\lambda$ as above. Then $\alpha$ is an element of $G$ of order 3 and $\alpha^{-1}(z)=\lambda(c(1-z))$.
\end{lemma}

\begin{proof}
Since $c$ is in $N$, $-c^{-1}$ is in $R$, and $\alpha$ is in $G$. Also $\alpha(0)=\infty$, $\alpha(\infty)=1$ and $\alpha(1)=1-c^{-1}\cdot c=0$. So $\alpha\circ\alpha\circ\alpha$ fixes the letters 0, $\infty$ and 1; by Lemma \ref{lemma2.4} (2), $\alpha$ has order 3. Finally if $\mu$ is the map $z\rightarrow \lambda(c(1-z))$, then $\mu(\alpha(z))=\lambda(\lambda(z))=z$, and so $\alpha^{-1}=\mu$.
\end{proof}

The proof of the classification theorem for $p\equiv 3\pod{4}$ now divides into 2 subcases. In this section we treat the ``main case'' where the $c$ of Corollary \ref{corollary4.2} is $-1$, showing that $n\equiv -1\pod{p-1}$ so that $\lambda(z)=-\frac{1}{z}$ for all $z$. The ``special case'', $c\ne -1$, which leads to conclusion (b) of the classification theorem will be handled in the next section --- it's a bit more technical.

\begin{lemma}
\label{lemma4.4}
In the main case the only solutions of $x^{n}=x$ in $(\newz/p)^{*}$ are $1$ and $-1$.
\end{lemma}

\begin{proof}
Since $c=-1$, $c^{-1}=-1$, and $\lambda(x)=-x^{n}$ for all $x$ in $(\newz/p)^{*}$. Thus $\alpha(x)=1-x^{n}$. Suppose now that $x\ne 1$ is in $(\newz/p)^{*}$ with $x^{n}=x$. Then $\alpha(x)=1-x$ and so $\alpha(\alpha(x))=1-(1-x)^{n}$. By Lemma \ref{lemma4.3}, $\alpha^{-1}(x)=\lambda(x-1)=-(x-1)^{n}=(1-x)^{n}$.  Since $\alpha(\alpha(x))=\alpha^{-1}(x)$, $(1-x)^{n}=\frac{1}{2}$. Raising to the $n$th power we find that $1-x=2^{-n}$. So $1$ and $1-2^{-n}$ are the only possible solutions of $x^{n}=x$ in $(\newz/p)^{*}$. Since $1$ and $-1$ are solutions we're done.
\end{proof}

\begin{proposition}
\label{prop4.5}
Suppose $p\equiv 3\pod{4}$. In the main case, $\lambda$ is the map $z\rightarrow -\frac{1}{z}$, and so $z\rightarrow -\frac{1}{z}$ is in $G$.
\end{proposition}

\begin{proof}
By Lemma \ref{lemma4.4} the only solution of $x^{n}=x$ in the cyclic group $R$ of order $\frac{p-1}{2}$ is $1$. So $n-1$ is prime to $\frac{p-1}{2}$. Now $\frac{p-1}{2}$ divides $(n+1)(n-1)$ by Lemma \ref{lemma2.6}. So it divides $n+1$, and as $n$ is odd, $n\equiv -1\pod{p-1}$. Then for $z$ in $(\newz/p)^{*}$, $\lambda(z)=-z^{n}=-\frac{1}{z}$. Furthermore $\lambda(0)=\infty$, $\lambda(\infty)=0$.
\end{proof}

\section{$\bm{p\equiv 3\pod{4}$}. The special case}
\label{section5}

We continue with the notation of Section \ref{section4} but now assume $c\ne -1$

\begin{lemma}
\label{lemma5.1}
Let $x$ be a power of $-c$, and suppose that $1-x$ is in $N$. Then:
\begin{enumerate}
\item[(a)] $\alpha(\alpha(x))=1-c^{-2}(1-x)^{n}$
\item[(b)] $\alpha(\alpha(x^{-1}))=1+x^{-1}(1-x)^{n}$
\item[(c)] $\alpha^{-1}(x)=c^{2}(1-x)^{n}$
\item[(d)] $\alpha^{-1}(x^{-1})=-x^{-1}(1-x)^{n}$
\end{enumerate}
\end{lemma}

\begin{proof}
Since $c^{n}=c$ and $n$ is odd, $x^{n}=x$.  Since $c$ is in $N$, $x$ is in $R$. Thus $\alpha(x)=1-c^{-1}(cx^{n})=1-x$, and similarly $\alpha(x^{-1})=1-x^{-1}=\frac{1-x}{-x}$, which is in $R$.

Now $\alpha(\alpha(x))=\alpha(1-x)=1-c^{-1}c^{-1}(1-x)^{n}$ giving (a). And $\alpha(\alpha(x^{-1}))=\alpha\left(\frac{1-x}{-x}\right)=1-\left(\frac{1-x}{-x}\right)^{n}=1+x^{-1}(1-x)^{n}$ giving (b). Furthermore $\alpha^{-1}(x)=\lambda(c(1-x))$. Since $c(1-x)$ is in $R$, this is $c\cdot c(1-x)^{n}$. Finally $\alpha^{-1}(x^{-1})=\lambda\left(\frac{c(1-x)}{-x}\right)=c^{-1}\left(\frac{c}{-x}\right)\cdot(1-x)^{n}=-x^{-1}(1-x)^{n}$.
\end{proof}

\begin{lemma}
\label{lemma5.2}
In the situation of Lemma \ref{lemma5.1}, $c^{2}+c^{-2}+2x^{-1}=0$.
\end{lemma}

\begin{proof}
$\alpha(\alpha(x))=\alpha^{-1}(x)$ by Lemma \ref{lemma4.3}. (a) and (c) above tell us that $(c^{2}+c^{-2})(1-x)^{n}=1$. Similarly, (b) and (d) tell us that $2x^{-1}(1-x)^{n}=-1$. Adding these identities and noting that $(1-x)^{n}\ne 0$ we get the result.
\end{proof}

\begin{lemma}
\label{lemma5.3}
$c^{3}=-1$, and either $c^{4}+3=0$ or $3c^{4}+1=0$.
\end{lemma}

\begin{proof}
There is an $x$ in $\{c^{2},c^{-2}\}$ such that $1-x$ is in $N$. For neither $1-c^{2}$ nor $1-c^{-2}$ is 0, and if both were in $R$, their quotient, $-c^{2}$, would be in $R$. Similarly there is a $y$ in $\{-c,-c^{-1}\}$ such that $1-y$ is in $N$. By Lemma \ref{lemma5.2}, $c^{2}+c^{-2}+2x^{-1}$ and $c^{2}+c^{-2}+2y^{-1}$ are both 0. So $x=y$, and $c^{3}=-1$. Also, since $c^{2}+c^{-2}+2x^{-1}=0$, either $c^{2}+3c^{-2}$ or $3c^{2}+c^{-2}$ is 0.
\end{proof}

\begin{proposition}
\label{prop5.4}
$p=7$. Furthermore either $c=3$ and $\lambda=(0 \infty)(1 3)(2 6)(4 5)$, or $c=5$ and $\lambda=(0 \infty)(1 5)(2 3)(4 6)$
\end{proposition}

\begin{proof}
Suppose $c^{4}+3=0$. Then, since $c^{3}=-1$, $c=3$. Also $27=-1$ in $\newz/p$, and so $p=7$. We know that $c^{n}=c$ in $(\newz/p)^{*}$. Since $c=3$ is a generator of $(\newz/7)^{*}$, $z^{n}=z$ for all $z$ in $(\newz/7)^{*}$. In particular if $z$ is in $R$, $\lambda(z)=cz^{n}=3z$, and so $\lambda=(0 \infty)(1 3)(2 6)(4 5)$.

Suppose $3c^{4}+1=0$. Then since $c^{3}=-1$, $3c=1$. So $27c^{3}=1$, $-27=1$ in $\newz/p$, and once again $p=7$. Since $3c=1$, $c=5$. Arguing as in the paragraph above we find that $\lambda=(0 \infty)(1 5)(2 3)(4 6)$.
\end{proof}

Suppose now that $c=3$. Then $z\rightarrow z+1$ is in $G$, and since 2 is in $R$, $z\rightarrow 2z$ is also in $G$. To complete the proof of the classification theorem for $c=3$ it suffices to show that the group of permutations of $\newz/7\cup\{\infty\}$ generated by $z\rightarrow z+1$, $z\rightarrow 2z$ and $\lambda$ is of order 168, and has a normal subgroup of order 8 (since $G$ contains this group, and $|G|=168$).  This can be shown by brute force, but here's a conceptual argument using some of the theory of finite fields.

Let $F$ be the field of 8 elements, $\zeta$ be a generator of $F^{*}$, and $U$ be the group of permutations of $F$ generated by $x\rightarrow x+1$, $x\rightarrow \zeta x$ and $x\rightarrow x^{2}$. If $r$ is in $F^{*}$, the conjugate of $x\rightarrow x+1$ by $x\rightarrow rx$ is $x\rightarrow x+r$. It follows that $x\rightarrow x+1$ and $x\rightarrow \zeta x$ generate the ``affine group'' of $F$, a group of order $7\cdot 8=56$. Furthermore $x\rightarrow x^{2}$ is a permutation of $F$ of order 3 normalizing the affine group. We conclude that $|U|=56\cdot 3 = 168$. The translations $x\rightarrow x+a$ evidently form a normal subgroup of $U$ with 8 elements.

Now identify $F$ with $\newz/7\cup\{\infty\}$ by mapping 0 to $\infty$ and $\zeta^{i}$ to $i$. Then $U$ may be viewed as a group of permutations of $\newz/7\cup\{\infty\}$ of order 168. $x\rightarrow \zeta x$ is the permutation $z\rightarrow z+1$, while $x\rightarrow x^{2}$ is the permutation $z\rightarrow 2z$. Now $\zeta$ has degree 3 over $\newz/2$, and so $\zeta^{3}+\zeta+1=0$ or $\zeta^{3}+\zeta^{2}+1=0$. Choose $\zeta$ so that $\zeta^{3}+\zeta+1=0$. Then, $1+1=0$, $1+\zeta=\zeta^{3}$, $1+\zeta^{2}=\zeta^{6}$ and $1+\zeta^{4}=\zeta^{12}=\zeta^{5}$. So $x\rightarrow x+1$ is the permutation $(0 \infty)(1 3)(2 6)(4 5)$ of $\newz/7\cup\{\infty\}$. Thus the group generated by $z\rightarrow z+1$, $z\rightarrow 2z$ and $(0 \infty)(1 3)(2 6)(4 5)$ identifies with $U$, and has order 168, and a normal subgroup of order 8. The argument is the same when $c=5$, except that we now take $\zeta$ with $\zeta^{3}+\zeta^{2}+1=0$.

\section{Simplicity results for $\bm{PSL_{2}(F)}$ and final remarks}
\label{section6}

The simplicity result of Galois has been generalized in various ways. For example if $F$ is any field with more than 3 elements, finite or infinite, then $PSL_{2}(F)$ is a simple group. I'll give one of the many proofs of this result. Let $N$ be a normal subgroup of $SL_{2}(F)$ containing some non-scalar matrix. If suffices to show that $N=SL_{2}(F)$.

\begin{lemma}
\label{lemma6.1}
There is an $\left(\begin{smallmatrix}a&b\\c&d\end{smallmatrix}\right)$ in $N$ with $b\ne 0$.
\end{lemma}

\begin{proof}
If not, then since $N$ is normal, every element of $N$ also has $c=0$, and so is diagonal. But if $\left(\begin{smallmatrix}a&0\\0&d\end{smallmatrix}\right)$ is a non-scalar element of $N$, the conjugate of $\left(\begin{smallmatrix}a&0\\0&d\end{smallmatrix}\right)$ by $\left(\begin{smallmatrix}1&1\\0&1\end{smallmatrix}\right)$ isn't diagonal.
\end{proof}

Now let $P$ and $P^{\prime}$ be the subgroups $\left(\begin{smallmatrix}1&0\\ *&1\end{smallmatrix}\right)$ and $\left(\begin{smallmatrix}1&*\\0&1\end{smallmatrix}\right)$ of $SL_{2}(F)$. $P$ and $P^{\prime}$ evidently generate $SL_{2}(F)$. Since $N$ is normal, $PN=NP$, and is the subgroup of $SL_{2}(F)$ generated by $P$ and $N$.

\begin{remark*}
If we can show that $N\supset P$, then since it is normal it also contains $P^{\prime}$, and so $N=SL_{2}(F)$.
\end{remark*}

\begin{lemma}
\label{lemma6.2}
$PN=NP=SL_{2}(F)$.
\end{lemma}

\begin{proof}
Take $\left(\begin{smallmatrix}a&b\\c&d\end{smallmatrix}\right)$ in $N$ as in Lemma \ref{lemma6.1}. Multiplying this matrix on the left by $\left(\begin{smallmatrix}1&0\\r&1\end{smallmatrix}\right)$ has the effect of adding $r\cdot$ (row 1) to row 2. So $PN$ contains a matrix $\left(\begin{smallmatrix}a&b\\ *&0\end{smallmatrix}\right)$. Multiplying this new matrix on the right by $\left(\begin{smallmatrix}1&0\\s&1\end{smallmatrix}\right)$ has the effect of adding $s\cdot$ (column 2) to column 1. So $PNP=PN$ contains a matrix $\left(\begin{smallmatrix}0&b\\ *&0\end{smallmatrix}\right)$. This matrix conjugates $P$ into $P^{\prime}$. So $PN$ contains $P^{\prime}$ as well as $P$, and is all of $SL_{2}(F)$.
\end{proof}

\begin{proposition}
\label{prop6.3}
If $|F|>3$, $N\supset P$. So by the remark above, $N=SL_{2}(F)$. Consequently $SL_{2}(F)$ is simple.
\end{proposition}

\begin{proof}
Take $a\ne 0$, $1$ or $-1$ in $F$ and let $d=a^{-1}$. By Lemma \ref{lemma6.2}, $\left(\begin{smallmatrix}a&0\\0&d\end{smallmatrix}\right)=\left(\begin{smallmatrix}1&0\\-r&1\end{smallmatrix}\right)\cdot B$ for some $B$ in the normal subgroup $N$. Then $B=\left(\begin{smallmatrix}1&0\\r&1\end{smallmatrix}\right) \left(\begin{smallmatrix}a&0\\0&d\end{smallmatrix}\right) = \left(\begin{smallmatrix}a&0\\ra&d\end{smallmatrix}\right)$. A short calculation shows that as $A$ runs over all the elements of $P$, $(ABA^{-1})\cdot B^{-1}$ also runs over all the elements of $P$. Since each $(ABA^{-1})\cdot B^{-1}$ is in $N$, we're done.
\end{proof}

In presenting the material of this note to a class one might add the following remarks:

\begin{enumerate}
\item[(1)] Let $F$ be the field of $q$ elements where $q$ is a prime power. Then $PSL_{2}(F)$ has order $\frac{q^{3}-q}{2}$ or $q^{3}-q$ according as $q$ is odd or even. By Proposition \ref{prop6.3} these groups are simple for $q>3$.
\item[(2)] If $F$ is the field of $q$ elements it's true that the ``only'' simple group having the same order as $PSL_{2}(F)$ is $PSL_{2}(F)$ itself. But I think that all proofs of this generalization of Frobenius' result are very difficult. The case $q=4$ is trivial --- since $4^{3}-4=\frac{5^{3}-5}{2}=60$, uniqueness when $q=4$ follows from the uniqueness when $q=5$. The next cases of interest are $q=9$ when $|G|=\frac{9^{3}-9}{2}=360$, and $q=8$ when $|G|=8^{3}-8=504$. In 1893, F. N. Cole, \cite{3} (best known to mathematicians for the establishment in his honor of the Cole prize), used intricate arguments to handle these cases. He starts by showing that $G$ is isomorphic to a doubly transitive permutation group on $q+1$ letters. But this is no longer an easy consequence of Sylow theory, as it is in the case of prime $q$.
\item[(3)] For $n>2$, let $SL_{n}(F)$ be the group of $n$ by $n$ determinant 1 matrices with entries in $F$, and $PSL_{n}(F)$ be the quotient of $SL_{n}(F)$ by the group of determinant 1 scalar matrices. It can be shown that for all $F$ and for all $n>2$ the group $PSL_{n}(F)$ is simple. But now the generalization of Frobenius' theorem has an exception. If $F$ is the field of 4 elements then $PSL_{3}(F)$ and $PSL_{4}(\newz/2)$ are non-isomorphic simple groups of order 20,160. (The group of even permutations of 8 letters is also simple of order 20,160, but it is isomorphic to $PSL_{4}(\newz/2)$.)
\end{enumerate}


\label{}



\end{document}